\documentclass[12pt,oneside,reqno]{amsart}

\usepackage{hyperref}

\usepackage[headsep=30pt,footskip=30pt,twoside=false,bottom=2.5cm]{geometry}

\usepackage[all]{xy}

\xymatrixcolsep{2cm}

\numberwithin{equation}{section}

\allowdisplaybreaks[1]

\usepackage{etoolbox}

\makeatletter

\patchcmd{\thesubsection}{\arabic}{\arabic}{}{}

\patchcmd{\@seccntformat}{\@secnumfont}{%
  \@secnumfont\expandafter\protect\csname format#1\endcsname}{}{}

\patchcmd{\@startsection}{\@afterindenttrue}{\@afterindentfalse}{}{}

\patchcmd{\subsection}{-.5em}{.3\linespacing}{}{}

\makeatother

\theoremstyle{plain}

\newtheorem{theorem}{Theorem}[section]

\newtheorem{proposition}[theorem]{Proposition}

\newtheorem{lemma}[theorem]{Lemma}

\newtheorem{corollary}[theorem]{Corollary}

\theoremstyle{remark}

\newtheorem{remark}[theorem]{Remark}
\newtheorem{question}[theorem]{Question}

\newcommand{\CG}[3][]{\ensuremath{\mathrm{CH}_{#1}^{#2}  (#3)}}

\newcommand{\Pic}[1]{\ensuremath{\mathrm{Pic} (#1)}}

\newcommand{\dv}[1]{\ensuremath{\mathrm{div} (#1)}}

\newcommand{\ad}[1]{\ensuremath{\mathrm{ad}  (#1)}}

\newcommand{\ENd}[1]{\ensuremath{\mathrm{End}  (#1)}}

\newcommand{\cat}[1]{\ensuremath{\mathcal{#1}}}

\newcommand{\at}[2][]{\ensuremath{\mathrm{at}_{#1} (#2)}}

\newcommand{\G}[2][]{\ensuremath{\mathbf{CH}_{#1} (#2)}}

\newcommand{\id}[1]{\ensuremath{\mathbf{1}_{#1}}}

\renewcommand{\dim}[2][]{\ensuremath{\mathrm{dim}_{#1}(#2)}}

\newcommand{\rk}[2][]{\ensuremath{\mathrm{rk}_{#1}(#2)}}

\newcommand{\h}{\ensuremath{\mathbf{H}}}

\newcommand{\Z}{\ensuremath{\mathbb{Z}}}

\newcommand{\A}{\ensuremath{\mathbb{A}}}

\newcommand{\Q}{\ensuremath{\mathbb{Q}}}

\newcommand{\R}{\ensuremath{\mathbb{R}}}

\newcommand{\p}{\ensuremath{\mathbb{P}}}

\newcommand{\C}{\ensuremath{\mathbb{C}}}

\newcommand{\tr}[1]{\ensuremath{\mathrm{Tr}(#1)}}

\newcommand{\struct}[1]{\ensuremath{\mathcal{O}_{#1}}}

  \newcommand{\RCOH}[4][]{\ensuremath{\mathrm{H}^{#1}_{#2}(#3,#4)}} 
  
\newcommand{\coh}[3]{\ensuremath{\mathrm{H}^{#1}(#2,#3)}}

\begin{document}

\title[ moduli spaces of holomorphic and logarithmic connections]{A note on the moduli spaces of
holomorphic and logarithmic  connections over a compact Riemann surface}

\author{Anoop Singh}

\address{Department of Mathematical Sciences, Indian Institute of Technology (BHU), Varanasi 221005, India}
  \email{anoopsingh.mat@iitbhu.ac.in}


\begin{abstract}
Let $X$ be a compact Riemann surface of genus $g \geq 3$.
We consider the moduli space of holomorphic connections
over $X$ and the moduli space of logarithmic connections 
singular over a finite subset of  $X$ with fixed 
residues. We determine the Chow group of these moduli spaces. We compute  the global sections of the sheaves of 
differential operators on ample line bundles and their symmetric powers 
over these moduli spaces, and show that  they are constant under certain 
conditions. We show the Torelli type 
theorem for the moduli space of logarithmic 
connections. We also describe the rational connectedness of these moduli spaces.
 \end{abstract}

\maketitle
\tableofcontents

\section{Introduction and statements of the results}

Let $X$ be a compact Riemann surface of genus 
$g \geq 3$. We consider the moduli space $\cat{M}_h(n)$ of rank $n$ holomorphic connections over $X$. In \cite{S1} and \cite{S2}, Simpson constructed the moduli space of holomorphic connections
over a smooth complex projective variety.

Let  $$S = \{x_1, \ldots, x_m \} $$  be a fixed subset of $X$  such 
that $x_i \neq  x_j$ for all $i \neq j$. We consider the 
moduli space $\cat{M}_{lc}(n,d)$ of logarithmic connections of rank $n$ and degree $d$, singular over 
$S$, with fixed residues (see section \ref{Pre} for the definition). The moduli space of logarithmic 
 connections over a complex projective variety singular
 over a smooth normal crossing divisor has been 
 constructed in \cite{N}. 
 
 \textbf{Assumption:} Throughout this article, we assume that the rank $n$ and degree $d$ are coprime, except for the case where we deal with the moduli space of holomorphic connections.
 
 Several algebro-geometric invariants like the Picard group,
 algebraic functions  of the 
 moduli space of holomorphic and logarithmic 
 connections have been studied, see \cite{B04}, \cite{BR1},
 \cite{BM7}, \cite{AS},  \cite{AS1}, and \cite{Se}.

 In the present article, our aim is to study the 
 Chow group, global sections of certain sheaves, Torelli type theorems, and rational connectedness of these moduli spaces.
 
The structure of the article is as follows.
In section \ref{Pre}, we define the notion of holomorphic and logarithmic connections in a holomorphic 
vector bundle over $X$, and recall their moduli spaces.

 In section \ref{Chow-conn}, we compute the  Chow group of the moduli spaces which is  motivated by the following result
 in \cite{CH}.
 Let $\cat{U}^s(2, \struct{X}(x_0))$ be the moduli space 
 of stable vector bundles of rank $2$ with determinant
 $\struct{X}(x_0)$, where $x_0 \in X$. Then,
 in \cite{CH}, the Chow group
of $1$-cycles on  $\cat{U}^s(2,\struct{X}(x_0))$ has been computed, and it is
proved that
\begin{equation}
\label{eq:0.12}
\CG[1]{\Q}{\cat{U}^s(2,\struct{X}(x_0))} \cong 
\CG[0]{\Q}{X}.
\end{equation}

Fix a holomorphic line bundle $L$ over $X$ of degree 
 $d$. Consider the moduli space of logarithmic connections $\cat{M}_{lc}(n,L)$ of rank $n$ and fixed 
 determinant $L$ as described in \eqref{eq:m2}.
 Let 
$\cat{M}_{lc}'(n,L) \subset \cat{M}_{lc}(n,L)$
be the moduli space of logarithmic connections $(E,D)$
with  $E$ stable as described in \eqref{eq:m3}.
Then, we show the following (see Theorem \ref{thm:2}).
For every $ 0 \leq l \leq  (n^2-1)(g-1)$, 
we have a canonical isomorphism 
\begin{equation}
\label{eq:0.1}
\G[l+(n^2-1)(g-1)]{\cat{M}_{lc}'(n,L)} \cong
\G[l]{\cat{U}^s(n,L)}.
\end{equation}
As a consequence for $n = 2$, we have (see Corollary \ref{cor:2}),
\begin{enumerate}
\item \label{o.1} $\G[3g-3]{\cat{M}_{lc}'(2,L)}
\cong \Z $.
\item \label{0.2} $\CG[3g-2]{\Q}{\cat{M}_{lc}'(2,L)}
\cong \CG[0]{\Q}{X}$.
\item \label{0.3} $\CG[6g-8]{\Q}{\cat{M}_{lc}'(2,L)}
\cong \CG[0]{\Q}{X} \oplus \Q$.
\end{enumerate}

Let $L_0$ be a holomorphic line bundle over $X$ of degree zero.
Let $\cat{M}'_h(n,L_0)$ and $\cat{U}^s(n,L_0)$ be the 
moduli space defined in \eqref{eq:m4} and \eqref{eq:3.15} respectively. Then, we show that (see Theorem \ref{thm:3}), 
for every $ 0 \leq l \leq  (n^2-1)(g-1)$, 
we have a canonical isomorphism 
\begin{equation}
\label{eq:0.24}
\CG[l+(n^2-1)(g-1)]{\Q}{\cat{M}_{h}'(n,L_0)} \cong
\CG[l]{\Q}{\cat{U}^s(n,L_0)}.
\end{equation}

In section \ref{Global-sec}, we study the global sections of certain locally free sheaves.
Let $\cat{M}_{lc}'(n,d)$ be the moduli space described in  \eqref{eq:m1}, and 
 $\zeta$ an ample 
 line bundle over $\cat{M}_{lc}'(n,d)$. For $k \geq 0$,
 let $\cat{D}^k(\zeta)$ denote the sheaf 
of differential operators on $\zeta$ of order $k$.
Consider the following natural morphism
\begin{equation}
\label{eq:0.21.5}
p_0: \cat{M}'_{lc}(n,d) \to \cat{U}^s(n,d)
\end{equation}
sending $(E,D)$ to $E$. Then 
we have a morphism 
\begin{equation}
\label{eq:0.41.15}
{\widetilde{p_0}}_{\sharp} : 
\coh{0}{T^*\cat{M}_{lc}'}{\struct{ T^*\cat{M}_{lc}'}}
\rightarrow \coh{0}{T^*\cat{U}^s(n,d)}{\struct{ T^*\cat{U}^s(n,d)}}.
\end{equation}
of vector spaces 
induced from $$\widetilde{p_0} : T^*\cat{U}^s(n,d) \to 
T^*\cat{M}'_{lc},$$ where $T^*\cat{U}^s(n,d)$ and 
$T^*\cat{M}'_{lc}$ are the cotangent bundles of $\cat{U}^s(n,d)$ and $\cat{M}'_{lc}(n,d)$ respectively.

Under the assumption that ${\widetilde{p_0}}_{\sharp}$
in \eqref{eq:0.41.15} is injective, we show that 
(see Theorem \ref{thm:4.1}),
for every $k \geq 0$,  
\begin{equation}
\label{eq:0.32}
 \coh{0}{\cat{M}_{lc}'(n,d)}{\cat{S}ym^k(\cat{D}^1 (\zeta))} = \C,
\end{equation}
and (see Proposition \ref{prop:4.1})
\begin{equation}
\label{0.pr}
 \coh{0}{\cat{M}_{lc}'(n,d)}{\cat{D}^k (\zeta)} = \C.
 \end{equation}
Under the same  assumption, the above result is true for 
the  moduli spaces $\cat{M}_{lc}'(n,L)$ (see \eqref{eq:m3}),
$\cat{M}'_h(n)$ (see \eqref{eq:m}) and $\cat{M}'_{h}(n,L_0)$
(see \eqref{eq:m4}).

In section \ref{Betti-mod-conn}, we prove the Torelli 
type result for the moduli space of logarithmic connections and change the notation to emphasis on $X$,
that is, 
$$
\cat{M}_{lc}(X) = \cat{M}_{lc}(X,S) :=
\cat{M}_{lc}(n,d)$$
and $$\cat{M}_{lc}(X,L)= \cat{M}_{lc}(X,S,L):= \cat{M}_{lc}(n,L).$$ First we show that 
the above moduli spaces do not depend on the choice of 
$S$ (see Lemma \ref{lem:5.1} and Lemma \ref{lem:5.2}), and therefore we remove $S$ from the notation.
We show the following (see Theorem \ref{thm:5.5}).

Let $(X,S)$ and $(Y, T)$ be two $m$-pointed compact Riemann surfaces of genus 
$g \geq 3$. Let $\cat{M}_{lc}(X,L)$ and $\cat{M}_{lc}(Y,L')$ be the corresponding moduli spaces of 
logarithmic connections. Then, $\cat{M}_{lc}(X,L)$ is isomorphic to $\cat{M}_{lc}(Y,L')$  if and only if $X$ is 
isomorphic to $Y$.

Next, we show the universal property of the morphism (see Proposition \ref{prop:5.7}) 
$$G : \cat{M}_{lc}(X) \longrightarrow Pic^d(X)$$
defined by sending $(E,D) \mapsto \bigwedge^nE$.
Thus, $\cat{M}_{lc}(X)$ determines the pair $(Pic^d(X), G)$   
up to an automorphism of $Pic^d(X)$.
In the end of section \ref{Betti-mod-conn}, we present 
a Torelli type theorem for $\cat{M}_{lc}(X)$, that is,
let $(X,S)$ and $(Y, T)$ be two $m$-pointed compact Riemann surfaces of genus 
$g \geq 3$. Let $\cat{M}_{lc}(X)$ and $\cat{M}_{lc}(Y)$ be the corresponding moduli spaces of 
logarithmic connections. Then, $\cat{M}_{lc}(X)$ is isomorphic to $\cat{M}_{lc}(Y)$  if and only if $X$ is 
isomorphic to $Y$.   

In the last section \ref{Rat}, we talk about rational
connectedness and rationality of the moduli space.
This section is motivated by the results in \cite{KS}. 
We show that the moduli spaces $\cat{M}_{lc}(n,d)$ and  $\cat{M}_h(n)$
are not rational (see Theorem \ref{thm:6.3}, and  Theorem \ref{thm:6.3.1}, respectively).
And finally we show that the moduli space $\cat{M}_{lc}(n,L)$ is rationally connected (see Corollary \ref{cor:6.7}).

\section{Preliminaries}
\label{Pre}

%
%

\subsection{Moduli spaces of holomorphic and logarithmic connections}
\label{Log-conn}
We recall the notion of  holomorphic and logarithmic connection
 in a 
holomorphic vector bundle over a smooth projective curve over $\C$, that is, over  
a compact Riemann surface.

Let $X$ be a compact Riemann surface of genus $g \geq 3$. 
Let $E$ be a holomorphic vector bundle over $X$.
A {\bf holomorphic connection} in $E$ is a $\C$-linear map
$$ D :  E \to E \otimes \Omega^1_{X}$$ 
which satisfies the Leibniz rule 
\begin{equation}
\label{eq:0.4}
D(f s)= f D(s) + df \otimes s,
\end{equation}
where $f$ is a local section of \struct{X} and $s$ is a 
local section of $E$.

A theorem due to Atiyah \cite{A} and Weil
\cite{W}, which is known 
as the \emph{Atiyah-Weil criterion}, says that a holomorphic vector bundle over a compact
Riemann surface admits a holomorphic connection if and only if the degree of each
indecomposable component of the holomorphic vector bundle is zero. 
Thus, if  $E$ admits a holomorphic connection, then
$$\deg{E} = 0.$$

The slope $\mu(E)$ of $E$ is defined as 
$$\mu(E) = \frac{\deg{E}}{\rk{E}}.$$ 

A holomorphic connection $D$ in $E$ is said to be
{\bf semistable} (respectively, {\bf stable}) if for every non-zero
proper subbundle $F$ of $E$ which is invariant under 
$D$, that is, 
$$D(F) \subset F \otimes \Omega^1_X,$$ we have,
\begin{equation*}
\mu(F) \leq 0 ~(\mbox{resp.}~ \mu(F) < 0),
\end{equation*}
where $\mu(E)$ denotes the 
slope of $E$. 

Let $\cat{M}_h(n)$ be the moduli space of semi-stable holomorphic 
connections of rank $n$. Then $\cat{M}_h(n)$ is a 
normal quasi-projective variety of dimension 
$2 n^2(g-1) +2$. Let 
$$\cat{M}_h^{sm}(n) \subset \cat{M}_h(n)$$ be the 
smooth locus of the variety.
Let 
\begin{equation}
\label{eq:m}
\cat{M}_h'(n) \subset \cat{M}_h^{sm}(n)
\end{equation}
 be the
open subvariety  
whose underlying vector bundle is stable. Then 
$\cat{M}_h'(n)$ is an irreducible smooth quasi-projective variety of the
same dimension as of $\cat{M}_h(n)$.

We  now define the logarithmic connection.
Fix a 
finite subset $$S = \{x_1, \ldots, x_m \} $$ of $X$  such 
that $x_i \neq  x_j$ for all $i \neq j$. 
Let   $$Z = x_1+ \cdots + x_m$$  denote the reduced effective divisor 
on $X$ associated to the finite set $S$. Let 
$\Omega^1_X(\log Z)$ denote the sheaf of 
logarithmic differential $1$-forms along $Z$,  see \cite{S} for more details. For the theory of 
the meromorphic and logarithmic connections, we refer to
two excellent sources \cite{D} and \cite{BM}.

A {\bf logarithmic connection}  on $E$ singular over $S$ is a $
\C$-linear map 
\begin{equation}
\label{eq:3}
D : E \to E \otimes \Omega^1_X(\log Z)  = E \otimes
\Omega^1_X \otimes \struct{X}(Z)  
\end{equation}
which satisfies the Leibniz identity
\begin{equation}
\label{eq:4}
D(f s)= f D(s) + df \otimes s,
\end{equation}
where $f$ is a local section of \struct{X} and $s$ is a 
local section of $E$.

A logarithmic connection $D$ in $E$ is said to be
{\bf semistable} (respectively, {\bf stable}) if for every non-zero
proper subbundle $F$ of $E$ which is invariant under 
$D$, that is, 
$$D(F) \subset F \otimes \Omega^1_X(\log Z),$$ we have,
\begin{equation*}
\mu(F) \leq \mu(E) (\mbox{resp.}~ \mu(F) < \mu(E)),
\end{equation*}
where $\mu(E)$ denotes the 
slope of $E$.

We next describe the notion of residues of a
logarithmic connection $D$ in $E$ singular over 
$S$.  We will denote the fibre of $E$ over any point $x 
\in X$ by $E(x)$.

Let $v \in E(x_\beta)$ be any vector in the fibre of $E$
over $x_\beta$. Let $U$ be an open set around $x_\beta$ and
$s: U \to E$ be a holomorphic section of $E$ over $U$ 
such that $s(x_\beta) = v$. Consider the following 
composition 
\begin{equation*}
\label{eq:5}
\Gamma(U,E) \to \Gamma(U, E \otimes\Omega^1_X \otimes 
\struct{X}(S)) \to E \otimes\Omega^1_X \otimes \struct{X}
(S)(x_\beta) = E(x_\beta),
\end{equation*}
where the equality is given because 
for any $x_\beta \in S$, the fibre $\Omega^1_X 
\otimes \struct{X}(S)(x_\beta)$ is canonically 
identified with $\C$ by sending a meromorphic form to 
its residue at $x_\beta$.
Then, we have an endomorphism on $E(x_\beta)$
sending $v$ to $D(s)(x_\beta)$. We need to check that 
this endomorphism is well defined. Let $s' : U  \to E$ be another holomorphic section  such that $s'(x_\beta) = v$. Then 
 $$(s - s')(x_\beta) = v - v = 0.$$
 Let $t$ be a local coordinate at $x_\beta$ on $U$ such that 
$t(x_\beta) = 0$, that is, the coordinate system $(U,t)$ is centered at $x_\beta$. 
Since  $ s - s' \in 
\Gamma(U,E)$ and  $(s - s')(x_\beta) = 0$, $s - s' = t  \sigma$ 
for some $\sigma \in \Gamma(U,E)$. Now,
\begin{align*}
D( s - s') = D(t \sigma)& = t D(\sigma) + dt \otimes \sigma \\
& = t D(\sigma) + t (\frac{dt}{t} \otimes \sigma),
\end{align*}
and hence $D( s - s')(x_\beta) = 0$, that is,
$D(s)(x_\beta) = D(s')(x_\beta)$.

 Thus, we have a well defined 
endomorphism, denoted by 
\begin{equation}
\label{eq:6}
Res(D,x_\beta) \in \ENd{E}(x_\beta) = \ENd{E(x_\beta)}
\end{equation}
that sends $v$ to $D(s)(x_\beta)$.  This endomorphism 
$Res(D,x_\beta)$ is called the \textbf{residue} of the logarithmic 
connection $D$ at the point $x_\beta \in S$ (see \cite{D} for the details). 

From \cite[Theorem 3]{O}, for a logarithmic connection
$D$ singular over $S$, we have
\begin{equation}
\label{eq:7}
\deg{E}  + \sum_{j=1}^m \tr{Res(D,x_j)} = 0,
\end{equation}
where $\deg{E}$ denotes the degree of $E$, and
$\tr{Res(D,x_j)}$ denotes the trace of the  
endomorphism $Res(D,x_j) \in \ENd{E(x_j)}$, for all 
$j =1, \ldots, m$.

Let $\mathbf{LC}(E)$ denote the set of all logarithmic 
connections in $E$ singular over $S$.
Then $\mathbf{LC}(E)$ is an affine space modelled 
over the vector space $\coh{0}{X}{\ENd{E} \otimes
\Omega^1_{X}(\log Z)}$, that is, if $D$ is any logarithmic connection in $E$ singular over $S$,
then 
\begin{equation*}
\mathbf{LC}(E) = D + \coh{0}{X}{\ENd{E} \otimes
\Omega^1_{X}(\log Z)}.
\end{equation*}
Recall that an endomorphism $\phi \in \ENd{E(x_j)}$
is said to be a rigid endomorphism if for every 
global endomorphism $\alpha \in \coh{0}{X}{\ENd{E}} $
we have 
$$\phi \circ \alpha(x_j) = \alpha(x_j) \circ \phi,$$
where $\alpha(x_j) : E(x_j) \to E(x_j)$ is an endomorphism.

In what follows, 
we fix a rigid endomorphism $\Phi_j \in \ENd{E(x_j)}$, 
for every $j = 1, \ldots, m$, such that for 
every direct summand $F \subset E$, we have 
\begin{equation}
\label{eq:8}
\deg{F} + \sum_{j = 1}^m \tr{\Phi_j\vert_{F(x_j)}} = 0.
\end{equation}
Here $\tr{\Phi_j\vert_{F(x_j)}}$ makes sense, because
from \cite[Theorem 1.3 (1)]{B},
for a rigid endomorphism $\Phi_j \in \ENd{E(x_j)}$, and for every direct summand $F$ of $E$,
we have $$\Phi_j(F(x_j)) \subset F(x_j).$$

Let $\mathbf{LC}(E; \Phi_1,\ldots, \Phi_m )$
denote the set of all logarithmic connections singular 
over $S$ with fixed residues $\Phi_j$ for all
$j = 1, \ldots, m$, that is,

$$\mathbf{LC}(E; \Phi_1, \ldots, \Phi_m)$$ 
$$= \{ D~\vert~D \mbox{ is a logarithmic connection in}~E~
\mbox{with}~Res(D,x_j) = \Phi_j~\mbox{for all}~ j = 1, \ldots, m  \}.$$
Then, 
$\mathbf{LC}(E;\Phi_1, \ldots, \Phi_m)$ is an 
affine space modelled over $\coh{0}{X}{\Omega^1_X \otimes \ENd{E}}$.

%
Notice the difference between vector spaces when 
residue is fixed and otherwise.

We impose some more conditions on the residues 
$\Phi_j$ for $1 \leq j \leq m$  to get a `well behaved' 
moduli space of logarithmic connections singular over 
$S$ with fixed residues.

Suppose that the residues (rigid endomorphisms) $\Phi_j$
for every $j = 1, \ldots, m$ satisfy the following
condition.
\begin{enumerate}
\item[(P1):] \label{P} For every non-zero subbundle $F \subset E$, $$\Phi_j(F(x_j)) \subset F(x_j), $$
and 
$$ \frac{\tr{\Phi_j
\vert_{F(x_j)}}}{\rk{F}} = \frac{\tr{\Phi_j}}{\rk{E}}.$$
\end{enumerate}


If we take $\Phi_j = \alpha_j \id{E(x_j)}$, where 
$\alpha_j \in \C$ and 
$\id{E(x_j)}$ is the identity morphism on $E(x_j)$,
for every $1 \leq j \leq m$,
then $\{\Phi_j\}_{1 \leq j \leq m}$ satisfies 
\mbox{(P1)}.  In what follows, we take $\Phi_j =
\alpha_j \id{E(x_j)}$ for every $j =1, \ldots, m$.

We have an easy result.

\begin{lemma}
\label{lem:2}
Let $D \in \mathbf{LC}(E;\Phi_1, \ldots, \Phi_m)$ with
$\{\Phi_j\}_{1 \leq j \leq m}$ satisfying \mbox{(P1)}.
Then $D$ is semi-stable. Moreover, if $(\deg E, \rk E) = 1$, then $D$ is stable.
\end{lemma}

%

A logarithmic connection $D$ in a holomorphic vector 
bundle $E$ 
is called \textbf{irreducible}  if for any 
holomorphic  subbundle $F$
of $E$ with 
$D(F) \subset F \otimes \Omega^1_X(\log Z)$, then either 
$F = E$ or 
$F = 0$.

If $D \in \mathbf{LC}(E;\Phi_1, \ldots, \Phi_m)$
satisfies \mbox{(P1)}, and $(\deg{E}, \rk{E}) = 1$,
then $D$ is irreducible.

Let $\cat{M}_{lc}(n,d)$ denote the moduli space which 
parametrizes the isomorphic class of  pairs $(E,D)$,
where, by a pair $(E,D)$ we mean that
\begin{enumerate}
\item \label{a} $E$ is a holomorphic vector bundle of 
rank $n$ and degree $d$ over $X$, such that 
$(n , d) = 1$.
\item \label{b} $D$ is a logarithmic connection with fixed 
residues $Res(D, x_j) = \Phi_j$ satisfying
\mbox{(P1)}.
\end{enumerate}
Two pairs $(E, D)$ and $(E', D')$ satisfying the above 
conditions \eqref{a} and \eqref{b} are said to be isomorphic if 
there exists an isomorphism $\Psi : E \to E'$
such that the following diagram
\begin{equation*}
\xymatrix{
E \ar[d]^{\Psi} \ar[r]^D & E \otimes \Omega^1_X(\log Z) \ar[d]^{\Psi \otimes \id{\Omega^1_X(\log Z)}} \\
E' \ar[r]^{D'} & E' \otimes \Omega^1_X(\log Z)\\
}
\end{equation*}
commutes.

From \cite[Theorem 3.5]{N}, the moduli space  $\cat{M}_{lc}(n,d)$ is a separated 
quasi-projective 
scheme over $\C$. 
Since $\{\Phi_j \}_{1 \leq j \leq m}$ satisfies 
\eqref{eq:8},  from \cite[Theorem 1.3 (2)]{B} the moduli space 
$\cat{M}_{lc}(n,d)$ is non-empty.

As we have observed  that every logarithmic connection 
$(E,D)$ in $\cat{M}_{lc}(n,d)$ is irreducible, and 
the singular points of  $\cat{M}_{lc}(n,d)$ 
correspond to reducible logarithmic connections 
\cite[p.n. 790]{BR1}, the moduli
space $\cat{M}_{lc}(n,d)$ is smooth. 
Since genus $g$ of $X$ is greater than or equal to $3$,  the moduli space 
$\cat{M}_{lc}(n,d)$ is irreducible \cite[Theorem 11.1]{S2}. 

Altogether, $\cat{M}_{lc}(n,d)$ is an irreducible smooth
quasi-projective variety of dimension $2 n^2(g-1) + 2$.
Let 
\begin{equation}
\label{eq:m1}
\cat{M}'_{lc}(n,d) \subset \cat{M}_{lc}(n,d)
\end{equation}
be the 
 moduli space of logarithmic connections whose underlying 
 vector bundles are stable. Then, from \cite[p.635, Theorem 2.8(A)]{Ma}  $\cat{M}'_{lc}(n,d)$ 
 is  an open subset of $\cat{M}_{lc}(n,d)$, and hence 
 an irreducible smooth
quasi-projective variety of dimension $2 n^2(g-1) + 2$.

Fix a holomorphic line bundle $L$ over $X$ of degree 
$d$, and a logarithmic connection $D_{L}$ on $L$ 
singular over $S$ with residues $Res(D_L, x_j) = 
\tr{\Phi_j}$ for all $j = 1, \ldots, m$. 
Let 
\begin{equation}
\label{eq:m2}
\cat{M}_{lc}(n,L) \subset \cat{M}_{lc}(n,d)
\end{equation}
be 
the moduli space  parametrising isomorphism class of 
pairs $(E,D)$ such that $$(\bigwedge^n E, \tilde{D}) 
\cong (L,D_L),$$ where $\tilde{D}$ is the logarithmic 
connection on $\bigwedge^n E$ induced by $D$. Then,
 $\cat{M}_{lc}(n,L)$ is an irreducible smooth 
 quasi-projective variety of dimension 
 $2 (n^2-1)(g-1)$.

Let 
\begin{equation}
\label{eq:m3}
\cat{M}_{lc}'(n,L) \subset \cat{M}_{lc}(n,L)
\end{equation}
be the moduli space of logarithmic connections $(E,D)$
with  $E$ stable.

\section{Chow group of the moduli spaces}
\label{Chow-conn}
In this section, we determine the Chow groups of the 
moduli spaces $\cat{M}'_{lc}(n,L)$, $\cat{M}'_{lc}(n,d)
$, $\cat{M}'_h(n)$ and $\cat{M}'_h(n,L_0)$.

Before that we recall the definition of 
Chow group of a quasi-projective scheme over a field
(see \cite{F} and \cite{Vo2}).

Let $\cat{X}$ be a quasi-projective scheme over a field $K$.
Let $Z_k(\cat{X})$ be  the free abelian group generated 
by the reduced and irreducible $k$-dimensional closed 
subvarieties of $\cat{X}$, or we can say the 
free abelian group generated by  $k$-dimensional 
closed integral subschemes of $\cat{X}$.
An element of $Z_k(\cat{X})$ is called a $k$-dimensional 
algebraic cycle on $\cat{X}$.

Let $f \in K(\cat{X})^*$. Then, we have a divisor
$\dv{f}$ on $\cat{X}$ associated to the non-zero 
rational function $f$ on $\cat{X}$.

A $k$-cycle $\alpha$ is 
{\bf rationally equivalent to zero},
written $\alpha \sim 0$, if there is a  finite number 
of $(k+1)$-dimensional subvarieties (that is, closed integral subschemes) $W_i$ of 
$\cat{X}$ and $f_i \in K(W_i)^*$, such that 
\begin{equation*}
\alpha = \sum_i \dv{f_i}.
\end{equation*}
Since $0 = \dv{1}$ and $\dv{f^{-1}}  = - \dv{f}$,
the cycles rationally equivalent to zero form a 
subgroup $Z_k(\cat{X})_{rat}$ of $Z_k(\cat{X})$.

We define the quotient group $$\G[k]{\cat{X}} :=  Z_k(\cat{X}) / 
Z_k(\cat{X})_{rat},$$ and call it the Chow group of $k$-cycles 
on $\cat{X}$. A graded sum is denoted by
$$\G[*]{\cat{X}} = \bigoplus_{k = 0}^{\dim{\cat{X}}} \G[k]{\cat{X}}.$$
The Chow group of $k$-cycles on $\cat{X}$ with rational 
coefficients will be denoted by $\CG[k]{\Q}{\cat{X}}$.

Let $\cat{U}^s(n, L)$ be the moduli space of stable 
vector bundles of rank $n$ with $\bigwedge^n E \cong L$. Then 
$\cat{U}^s(n,L)$ is a smooth projective variety 
of dimension $(n^2-1)(g-1)$, as we have assumed 
$n$ and $\deg(L) = d$ are coprime.

Let $x_0 \in X$, and $\struct{X}(x_0)$ the line
bundle on $X$ associated with the reduced effective divisor $x_0$.
For $n = 2$, we have 
 $\cat{U}^s(2,\struct{X}(x_0))$ the moduli space of stable 
 vector bundles of rank $2$ over $X$ whose determinant 
 is $\struct{X}(x_0)$.

In \cite{BKN}, it was shown that 
\begin{equation}
\label{eq:11}
\CG[3g-5]{\Q}{\cat{U}^s(2,\struct{X}(x_0))} \cong
\CG[0]{\Q}{X} \bigoplus \Q.
\end{equation}

In \cite{CH}, Choe and Hwang computed the Chow group
of $1$-cycles on  $\cat{U}^s(2,\struct{X}(x_0))$, and they 
proved that
\begin{equation}
\label{eq:12}
\CG[1]{\Q}{\cat{U}^s(2,\struct{X}(x_0))} \cong 
\CG[0]{\Q}{X}.
\end{equation}

Let $M$ be a holomorphic line bundle over $X$ of 
degree $d'$ and 
$$M_0 = \struct{X}(x_0) \otimes M^{\otimes 2}.$$
Then $\deg(M_0) = 2 d' + 1$.
Define a map 
\begin{equation*}
\Psi_M : \cat{U}^s(2, \struct{X}(x_0)) \to \cat{U}^s(2, M_0)
\end{equation*} 
by 
\begin{equation*}
\Psi_M ([E]) = [E \otimes M].
\end{equation*}
The map is well-defined, and in fact, an isomorphism
of varieties.

Thus, the above results $\eqref{eq:11}$ and 
$\eqref{eq:12}$ are true for $\cat{U}(2, L)$, where
$\deg(L)$ is odd.

Define
\begin{equation}
\label{eq:12.5}
p: \cat{M}_{lc}'(n, L) \to \cat{U}^s(n, L)
\end{equation}
by  sending $(E, D) \mapsto E$, that is,
$p$ is the forgetful map which forgets its 
logarithmic structure.

For every $E \in \cat{U}(n,L)$, $p^{-1}(E)$ is an 
affine space modelled over 
 $\coh{0}{X}{\Omega^1_X \otimes \ad{E}}$,
 where $\ad{E} \subset \ENd{E}$ is the 
 subbundle consisting of endomorphisms of $E$ whose trace 
 is zero. Actually, 
 $p$ is a fibre bundle and  using Riemann-Roch theorem and Serre duality
 it can be easily computed that the dimension of 
 $p^{-1}(E)$ is $(n^2-1)(g-1)$.

Let $\Omega^1_{\cat{U}^s(n,L)}$ denote the holomorphic cotangent bundle on $\cat{U}^s(n,L)$. Since $n$ and $d$ are coprime, there exists a universal bundle $\cat{E}$ on $\cat{U}^s(n,L) \times X$. Let $p_1 : \cat{U}^s(n,L) \times X \to \cat{U}^s(n,L)$ and $p_2 : \cat{U}^s(n,L) \times X \to X$ be the first and second projection
respectively. Then, from infinitesimal deformation theory  and local property of moduli space, we have 
 $R^1 p_1 (\ad{\cat{E}}) = T \cat{U}^s(n,L)$. Moreover,
 we have ${p_1}_*(\ad{\cat{E}} \otimes p_2^*(\Omega^1_X)) = \Omega^1_{\cat{U}^s(n,L)}$. 
Now, from the fact that $\mathbf{LC}(E;\Phi_1, \ldots, \Phi_m)$ is an 
affine space modelled over $\coh{0}{X}{\Omega^1_X \otimes \ENd{E}}$ as observed above, we have 

\begin{lemma}
\label{lem:1}
 $\cat{M}_{lc}'(n,L)$ is an 
$\Omega^1_{\cat{U}^s(n,L)}$-torsor over $\cat{U}^s(n,L)$.
\end{lemma}
 
We state two standard lemmas from the theory of 
Chow groups which we will use to compute the
Chow groups of moduli spaces.

Let $Y$ be a variety over a field $K$.
Let $i: F \to Y$ be the inclusion of a closed subscheme.
Let $j : U = Y \setminus F \to Y$ be the inclusion 
of the complement.
Since $j$ is an open immersion, it is flat, and 
$i$ is a closed immersion, it is proper.
Therefore, we have morphisms 
 $j^*: \G[k]{Y} \to \G[k]{U}$ and 
$i_*: \G[k]{F} \to \G[k]{Y}$ of Chow groups and 
$j^* \circ i_* = 0$, since the cycles supported on 
$F$ do not intersect $U$.
Thus, we have what is called {\bf localisation sequence}.
\begin{lemma}\cite[Lemma 9.12]{Vo2}
\label{lem:3}
The following sequence of abelian groups 
\begin{equation}
\label{eq:14}
\G[l]{F} \xrightarrow{i_*} \G[l]{Y} \xrightarrow{j^*}
\G[l]{U} \rightarrow 0.
\end{equation}
is exact for every $l = 0, \ldots, \dim{Y}$.
\end{lemma}

\begin{theorem}\cite[Theorem  9.25]{Vo2}
\label{thm:1}
Let $\pi: \p(E) \to Y$ be a projective bundle with 
rank $\rk{E} = r$.
Then the map
\begin{equation}
\label{eq:15}
\bigoplus_{k = 0}^{r-1} h^k \pi^* :
\bigoplus_{k = 0}^{r-1} \G[l-r+1+k]{Y} \to \G[l]{\p(E)}
\end{equation}
is an isomorphism, where $h \in \Pic{\p(E)}$ denotes the class of the tautological line 
bundle $\struct{\p(E)}(1)$.
\end{theorem}


\begin{theorem}
\label{thm:2} 
For every $ 0 \leq l \leq  (n^2-1)(g-1)$, 
we have a canonical isomorphism 
\begin{equation}
\label{eq:15.1}
\G[l+(n^2-1)(g-1)]{\cat{M}_{lc}'(n,L)} \cong
\G[l]{\cat{U}^s(n,L)}.
\end{equation}
\end{theorem}
\begin{proof}
  
Let $\cat{G}$ be an affine bundle modelled on a vector bundle $\cat{H}$ of rank $r = (n^2-1)(g-1)$ over $\cat{U}^s(n,L)$. Now, using the standard inclusion of the affine group in $\text{GL}(r+1, \C)$, we obtain a 
vector bundle $\cat{F}$ of rank $r +1$ together with an 
embedding of $\cat{G}$ in $\p (\cat{F})$ as an open 
subset with complement $\p (\cat{H}^ \vee)$.

Since $\cat{M}_{lc}'(n,L)$ is an 
$\Omega^1_{\cat{U}^s(n,L)}$-torsor over $\cat{U}^s(n,L)$ (see Lemma \ref{lem:1}), the above construction gives
an algebraic vector bundle $\cat{F}$
over $\cat{U}^s(n,L)$ with $\cat{M}'_{lc}(n,L)$  embedded 
in $\p (\cat{F})$ such that the complement $\p(\cat{F}) \setminus \cat{M}'_{lc}(n,L)$ is a hyperplane $\h$
at infinity. Now,
the hyperplane at infinity $\h$ is canonically identified with the 
total space of the projective bundle 
$\p (T \cat{U}^s(n,L))$, the space of all hyperplanes
in the fibre of the tangent bundle $T \cat{U}^s(n,L)$.

Putting $F = \p (T \cat{U}^s(n,L))$, $Y = \p (\cat{F})$ and $U = \cat{M}_{lc}'(n,L)$ in Lemma \ref{lem:3}, we get an
exact sequence 

 \begin{equation}
\label{eq:17}
\G[l]{\p (T \cat{U}^s(n,L))} \xrightarrow{i_*} \G[l]{\p (\cat{F}) } \xrightarrow{j^*}
\G[l]{\cat{M}_{lc}'(n,L)} \rightarrow 0.
\end{equation}
of abelian groups for every $l = 0, \ldots, \dim{\p (\cat{F})} = 2(n^2-1)(g-1)$.

Since $\rk{\cat{F}} = (n^2-1)(g-1)+1$, and $\rk{T \cat{U}^s
(n,L)} = (n^2-1)(g-1)$, from Theorem \ref{thm:1}, we 
have following isomorphisms 
\begin{equation}
\label{eq:18}
\G[l]{\p (\cat{F})} \cong \bigoplus_{k = 0}^{(n^2-1)(g-1)} \G[l-(n^2-1)(g-1)+k]{\cat{U}^s(n,L)}
\end{equation}
and 
\begin{equation}
\label{eq:19}
\G[l]{\p (T \cat{U}^s(n,L))} \cong \bigoplus_{k = 0}^{(n^2-1)(g-1) - 1} \G[l-(n^2-1)(g-1)+1+k]{\cat{U}^s(n,L)}.
\end{equation}

From \eqref{eq:17}, \eqref{eq:18} and \eqref{eq:19},
we get an exact sequence
\begin{equation}
\label{eq:20}
\begin{aligned}
& \bigoplus_{k = 0}^{(n^2-1)(g-1) - 1} \G[l-(n^2-1)(g-1)+1+k]{\cat{U}^s(n,L)} \xrightarrow{i_*} \\
& \bigoplus_{k = 0}^{(n^2-1)(g-1)} \G[l-(n^2-1)(g-1)+k]{\cat{U}^s(n,L)}
 \xrightarrow{j^*}
\G[l]{\cat{M}_{lc}'(n,L)} \rightarrow 0,
\end{aligned}
\end{equation}
which is actually a short exact sequence,
because $i_*$ is injective.
Thus, we have
\begin{equation}
\label{eq:21}
\G[l]{\cat{M}_{lc}'(n,L)} \cong
\G[l-(n^2-1)(g-1)]{\cat{U}^s(n,L)},
\end{equation}
for every $ (n^2-1)(g-1) \leq l \leq 2 (n^2-1)(g-1)$.
Now, rescaling $l$, we will get the desired result,
and this completes the proof.
\end{proof}

\begin{corollary}
\label{cor:1} For $l = 2(n^2-1)(g-1) -1$, we have
$$\G[l]{\cat{M}_{lc}'(n,L)} \cong \Z.$$
\end{corollary}
\begin{proof}
See \cite[Proposition 5.3]{AS}.
\end{proof}

\begin{corollary}
\label{cor:2} For $n = 2$, we have
\begin{enumerate}
\item \label{1} $\G[3g-3]{\cat{M}_{lc}'(2,L)}
\cong \Z $.
\item \label{2} $\CG[3g-2]{\Q}{\cat{M}_{lc}'(2,L)}
\cong \CG[0]{\Q}{X}$.
\item \label{3} $\CG[6g-8]{\Q}{\cat{M}_{lc}'(2,L)}
\cong \CG[0]{\Q}{X} \oplus \Q$.
\end{enumerate}
\end{corollary}
\begin{proof}
From Theorem \ref{thm:2}, and equations \eqref{eq:11} and 
\eqref{eq:12}, we conclude the Corollary.
\end{proof}

Next, let $\cat{U}^{s}(n,d)$ be the moduli space of 
stable vector bundle of rank $n$ and degree $d$.
 Consider the following natural morphism
\begin{equation}
\label{eq:21.5}
p_0: \cat{M}'_{lc}(n,d) \to \cat{U}^{s}(n,d)
\end{equation}
sending $(E,D)$ to $E$. Then $p_0^{-1}(E)$ is an affine 
space modelled over the vector space  $\coh{0}{X}{\Omega^1_X \otimes \ENd{E}}$. Since $E$ is stable,
the dimension of the vector space $\coh{0}{X}{\Omega^1_X \otimes \ENd{E}}$ is $n^2(g-1) +1$. In view of \cite[Theorem 1.1]{AS}, we can show a result similar 
to the Theorem \ref{thm:2}, which interprets the Chow groups of $\cat{M}_{lc}'(n,d)$ in terms of Chow groups of 
$\cat{U}^{s}(n,d)$.
\begin{theorem}
\label{thm:2.5}
For every $ 0 \leq l \leq  n^2(g-1) + 1$, 
we have canonical isomorphism 
\begin{equation}
\label{eq:21.6}
\G[l+n^2(g-1)+1]{\cat{M}_{lc}'(n,d)} \cong
\G[l]{\cat{U}^s(n,d)}.
\end{equation}
\end{theorem}

Now, we compute the same for the moduli space of holomorphic connections.
Fix a holomorphic line bundle
$L_0$ of  degree $0$ on $X$.
Let $\cat{U}(n,L_0)$ denote the moduli 
space of $S$-equivalence classes of semistable  vector bundles of rank $n$ and 
determinant $\bigwedge^n E \cong L_0 $. Then 
the moduli space $\cat{U}(n,L_0)$ is known 
to be an irreducible 
normal projective variety of dimension $(n^2-1)(g-1)$.

Let 
\begin{equation}
\label{eq:3.15}
\cat{U}^{s}(n, L_0) \subset 
\cat{U}(n, L_0)
\end{equation}
 be the open subvariety 
parametrizing the stable bundles on $X$. This open 
subvariety coincides with the smooth locus of 
$\cat{U}(n,L_0)$ follows from 
\cite[p.~20, Theorem 1]{NR}. 

Fix a holomorphic connection $D^{L_0}$ on $L_0$.
Let $\cat{M}_h(n, L_0)$ be the moduli space of 
holomorphic connections parametrising the isomorphism 
classes of the pairs $(E, D)$ where 
$E$ is a holomorphic vector bundle of rank $n$ with
$$(\bigwedge^n E, \tilde{D}) \cong (L_0, D^{L_0}),$$ and 
$\tilde{D}$ is a holomorphic connection on 
$\bigwedge^n E$ induced from $D$.
Then $\cat{M}_h(n, L_0)$ is an irreducible 
normal quasi-projective variety of dimension 
$2(n^2-1)(g-1)$.
Let $$\cat{M}_h^{sm}(n,L_0) \subset \cat{M}_h(n, L_0)$$
be the smooth locus of $\cat{M}_h(n, L_0)$.
Let 
\begin{equation}
\label{eq:m4}
\cat{M}_h'(n,L_0)  \subset \cat{M}^{sm}_h(n, L_0)
\end{equation}
be the subset consisting of holomorphic connections 
whose underlying vector bundle is stable.
Then $\cat{M}_h'(n,L_0)$ is an irreducible smooth
quasi-projective variety of dimension $2(n^2-1)(g-1)$.
 
Let
\begin{equation}
\label{eq:22}
q : \cat{M}_h'(n,L_0) \to  \cat{U}^{s}(n, L_0)
\end{equation}
be the forgetful map which forgets the holomorphic connection.
Then for every $E \in \cat{U}^s(n,L_0)$, $q^{-1}(E)$
is  an affine space modelled over $\coh{0}{X}{\Omega^1_X \otimes \ad{E}}$.
In fact, $\cat{M}_h'(n,L_0)$ is an 
$\Omega^1_{\cat{U}(n,L_0)}$-torsor on $\cat{U}^{s}(n,L_0)$.
 

Let $Y$ be an $N$-dimensional smooth quasi-projective 
variety. Then,
 the Picard group $\Pic{Y} \otimes_{\Z} \Q $ can be identified with $\CG[N-1]{\Q}{Y}$. Thus, it is enough
to compute $\Pic{Y}$.

The morphism $q$ as defined in $\eqref{eq:22}$ 
induces a homomorphism
\begin{equation}
\label{eq:23}
q^* : \Pic{\cat{U}^{s}(n, L_0)} \to \Pic{\cat{M}_h'(n,L_0)}
\end{equation}
of Picard groups given by sending a line bundle $M$
to its pull-back $q^*M$.
Using the similar techniques as in 
\cite[Theorem 1.2]{AS}, we can show the following.

\begin{proposition}
\label{prop:4}
The homomorphism $q^* : \Pic{\cat{U}^{s}(n, L_0)} \to \Pic{\cat{M}_h'(n,L_0)}$ is an isomorphism of groups.
\end{proposition}

Since $\Pic{\cat{U}^s(n,L_0)} \cong \Z$ (see \cite{DN}), we have 
\begin{corollary}
\label{cor:3}
For $l = 2(n^2-1)(g-1) -1$, we have 
$$\G[l]{\cat{M}_{h}'(n,L_0)} \cong \Z.$$
\end{corollary}

Using the exactly similar steps as in Theorem \ref{thm:2},
we can prove the following. 

\begin{theorem}
\label{thm:3}
For every $ 0 \leq l \leq  (n^2-1)(g-1)$, 
we have canonical isomorphisms 
\begin{equation}
\label{eq:24}
\G[l+(n^2-1)(g-1)]{\cat{M}_{h}'(n,L_0)} \cong
\G[l]{\cat{U}^s(n,L_0)}.
\end{equation}
\end{theorem}

Next, let $\cat{U}^s(n) := \cat{U}^s(n,0)$ be the moduli space of stable 
bundles of rank $n$ and degree zero. Then $\cat{U}^s(n)$
is an irreducible smooth projective variety of dimension 
$n^2(g-1)+ 1$. Again, we have a natural morphism
\begin{equation}
\label{eq:24.1}
q_0 : \cat{M}_h'(n) \to \cat{U}^s(n)
\end{equation} 
of varieties
which forgets the holomorphic connection.
Using the same method as above, we have the following
theorem similar to the Theorem \ref{thm:2.5}.
\begin{theorem}
\label{thm:2.6}
For every $ 0 \leq l \leq  n^2(g-1) + 1$, 
we have canonical isomorphism 
\begin{equation}
\label{eq:21.7}
\G[l+n^2(g-1)+1]{\cat{M}_{h}'(n)} \cong
\G[l]{\cat{U}^s(n)}.
\end{equation}
\end{theorem}

\section{Differential operators on the 
moduli spaces}
\label{Global-sec}
In \cite{B02}, Biswas studied the global sections of 
sheaves of differential operators on an ample line bundle 
over a polarised abelian variety. Also,
in \cite{AS2}, Hitchin variety is defined and  global sections of the sheaf of $k$-th order differential operators, and 
symmetric powers of the sheaf of first order differential operators on 
a line bundle over a Hitchin variety have been studied.  The moduli space of stable vector 
bundles over a compact Riemann surface is an example of Hitchin variety. The moduli spaces of holomorphic and 
logarithmic connections are not Hitchin varieties. 
In this section, we 
study the global sections of certain sheaves over the 
four moduli spaces
 $\cat{M}_{lc}'(n,d)$, $\cat{M}'_{h}(n)$, $\cat{M}'_{lc}(n,L)$ and $\cat{M}'_{h}(n,L_0)$ which we have defined in previous sections.
 
 Let $\zeta$ be an ample 
 line bundle over $\cat{M}_{lc}'(n,d)$. Let $k \geq 0$ be an integer.
Recall that a  differential operator of order $k$
on $\zeta$ is a 
$\C$-linear 
map
\begin{equation}
\label{eq:25}
\theta: \zeta \to \zeta
\end{equation}
such that for every open subset $U$ of $\cat{M}_{lc}'(n,d)$ and for every 
$f \in \struct{\cat{M}_{lc}'(n,d)}(U)$, the bracket 
$$[\theta|_U,f]:\zeta|_U \to \zeta|_U$$ defined as 
\begin{equation*}
\label{eq:26}
[\theta|_U,f]_V(s) = \theta(f|_V s) - \theta|_V P_V(s)
\end{equation*}
is a differential operator of order $k-1$, for every open subset $V$ of 
$U$, and for all $ s \in \zeta(V)$,
where a differential operator of order zero on $\zeta$ is just a
$\struct{\cat{M}_{lc}'(n,d)}$-module homomorphism (see 
\cite{GD} and \cite{R} for the definition and 
properties of differential operators).

For $k \geq 0$, let $\cat{D}^k(\zeta)$ denote the sheaf 
of differential operators on $\zeta$ of order $k$.
In fact, $\cat{D}^k(\zeta)$ is a locally free sheaf
with 
 $\cat{D}^0(\zeta) = 
\struct{\cat{M}_{lc}'(n,d)}$. 
Given a first order differential operator $\theta$
on $\zeta$, we get a section of the tangent bundle 
$T \cat{M}_{lc}'(n,d)$ denoted by $\sigma_1(\theta)$,
where $\sigma_1$ is called the symbol of a first order
differential operator. For simplicity, we shall denote 
$T \cat{M}_{lc}'(n,d)$ by $T \cat{M}_{lc}'$.
Thus, consider the symbol operator
$\sigma_1 :\cat{D}^1(\zeta) \to T \cat{M}_{lc}'$.
This induces a morphism 
\begin{equation*}
\label{eq:27}
\cat{S}ym^k(\sigma_1): \cat{S}ym^k(\cat{D}^1 (\zeta))  \to 
\cat{S}ym^k (T\cat{M}_{lc}') 
\end{equation*}
of $k$-th symmetric powers. 
Now, because of the following composition 
\begin{equation*}
\struct{\cat{M}_{lc}'(n,d)} \otimes \cat{S}ym^{k-1}(\cat{D}^1 (\zeta)) \hookrightarrow 
\cat{D}^1 (\zeta) \otimes \cat{S}ym^{k-1} (\cat{D}^1 (\zeta)) \to 
\cat{S}ym^{k} (\cat{D}^1 (\zeta)),
\end{equation*}
we have 
\begin{equation}
\label{eq:28}
\cat{S}ym^{k-1}(\cat{D}^1 (\zeta))  \subset \cat{S}ym^k (\cat{D}^1 (\zeta))
~~~ \mbox{for all}~ k \geq 1.
\end{equation}

Thus, we get a short exact sequence
\begin{equation}
\label{eq:29}
0 \to \cat{S}ym^{k-1}(\cat{D}^1 (\zeta)) \to \cat{S}
ym^{k}(\cat{D}^1 (\zeta)) \xrightarrow{\cat{S}ym^k 
(\sigma_1)} \cat{S}ym^k (T \cat{M}_{lc}') \to 0.
\end{equation}
Thus, we have 
\begin{equation}
 \label{eq:30}
 \cat{S}ym^{k}(\cat{D}^1 (\zeta)) / \cat{S}ym^{k-1}
 (\cat{D}^1 (\zeta)) \cong \cat{S}ym^k (T\cat{M}_{lc}'
 )~~~ \mbox{for all}~ k \geq 1.
 \end{equation}
 
 From \eqref{eq:28}, we have the following 
  chain of $\C$-vector spaces
 
 \begin{equation}
 \label{eq:31}
 \coh{0}{\cat{M}_{lc}'(n,d)}{\struct{\cat{M}_{lc}'(n,d)}} \subset \coh{0}{\cat{M}_{lc}'(n,d)}{\cat{S}ym^1(\cat{D}^1 (\zeta))}
 \subset   \ldots
 \end{equation}

Consider  the following commutative diagram,
\begin{equation}
\label{eq:4.6}
\xymatrix{
T^*\cat{M}'_{lc} \ar[d]^{\pi'}  & \ar[l]_{\widetilde{p_0}} T^*\cat{U}(n,d) \ar[d]^{\pi} \\
\cat{M}'_{lc}(n,d) \ar[r]^{p_0} & \cat{U}(n,d)\\
}
\end{equation}
where $\pi$, $\pi'$ are the canonical projections and 
$\widetilde{p_0}$ is induced from 
$p_0$ as defined in \eqref{eq:21.5}. 
Thus, we have a morphism 
\begin{equation}
\label{eq:41.15}
{\widetilde{p_0}}_{\sharp} : 
\coh{0}{T^*\cat{M}_{lc}'}{\struct{ T^*\cat{M}_{lc}'}}
\rightarrow \coh{0}{T^*\cat{U}^s(n,d)}{\struct{ T^*\cat{U}^s(n,d)}}.
\end{equation}
of vector spaces 
induced from $\widetilde{p_0}$.

\begin{theorem}
\label{thm:4.1}  Suppose that ${\widetilde{p_0}}_{\sharp}$ in \eqref{eq:41.15} is an injective morphism. Then,
for every $k \geq 0$, we have 
\begin{equation}
\label{eq:32}
 \coh{0}{\cat{M}_{lc}'(n,d)}{\cat{S}ym^k(\cat{D}^1 (\zeta))} = \C.
\end{equation}

\end{theorem}
\begin{proof}
Let 
\begin{equation}
\label{eq:k}
\cat{M}^{lc}_X := \cat{M}_{lc}(1,d)
\end{equation}
 be the moduli space of rank one logarithmic connections singular over 
$S$, with fixed residues $\tr{\Phi_j}$ for every 
$j = 1, \ldots,  m$, for more details, see \cite{Se} and \cite{AS1}. Then there is a natural morphism of 
varieties
\begin{equation}
\label{mor}
\det : \cat{M}'_{lc}(n,d) \longrightarrow \cat{M}^{lc}_X
 \end{equation}
 sending $(E,D) \mapsto (\bigwedge^nE, \widetilde{D})$,
 where $\widetilde{D}$ is the induced logarithmic  connection on 
 $\bigwedge^nE$.
 For any pair $(L, \nabla) \in \cat{M}^{lc}_X$, 
 $${\det}^{-1}((L,\nabla)) = \cat{M}'_{lc}(n,L).$$
 From \cite[Theorem 2]{Se}, we have $$\coh{0}{\cat{M}^{lc}_X}{\struct{\cat{M}^{lc}_X}} = \C,$$
 and from \cite[Theorem 1.4]{AS}, we have 
 $$\coh{0}{\cat{M}_{lc}'(n,L)}{\struct{\cat{M}_{lc}'(n,L)}} = \C.$$
 Combining both the results and using \eqref{mor}, we have 
$$\coh{0}{\cat{M}_{lc}'(n,d)}{\struct{\cat{M}_{lc}'(n,d)}} = \C.$$ Thus from \eqref{eq:31}, it is enough to 
show that for every $k \geq 0$, the inclusion
$$\coh{0}{\cat{M}_{lc}'(n,d)}{\struct{\cat{M}_{lc}'(n,d)}} \to \coh{0}{\cat{M}_{lc}'(n,d)}{\cat{S}ym^k(\cat{D}^1 (\zeta))}$$ is an isomorphism.
From the isomorphism in \eqref{eq:30}, we have the following commutative diagram
\begin{equation}
\label{eq:cd1}
\xymatrix@C=3em{
0 \ar[r] & \cat{S}ym^{k-1}(\cat{D}^1 (\zeta)) \ar[d] \ar[r] & \cat{S}ym^k(\cat{D}^1 (\zeta))
\ar[d] \ar[r]^{\cat{S}ym^k(\sigma_1)} & \cat{S}ym^k(T\cat{M}_{lc}') \ar[d] \ar[r] & 0 \\
0 \ar[r] & \cat{S}ym^{k-1}(T \cat{M}_{lc}') \ar[r] & 
\frac{\cat{S}ym^k(\cat{D}^1 (\zeta))}{\cat{S}ym^{k-2}(\cat{D}^1 (\zeta))} \ar[r] & \cat{S}ym^k(T\cat{M}_{lc}') \ar[r] & 0 
}
\end{equation}
which gives rise to the following commutative 
diagram of long exact sequences
\begin{equation}
\label{eq:cd2}
\xymatrix@C=2em{
\cdots \ar[r] & \coh{0}{\cat{M}_{lc}'(n,d)}{\cat{S}ym^k T
\cat{M}_{lc}'} \ar[d] \ar[r]^{\delta'_k} & \coh{1}
{\cat{M}_{lc}(n,d)'}{\cat{S}ym^{k-1}(\cat{D}^1 (\zeta))} \ar[d] 
\ar[r] & \cdots  \\
\cdots \ar[r] & \coh{0}{\cat{M}_{lc}'(n,d)}{\cat{S}ym^k T\cat{M}_{lc}'}       \ar[r]^{\delta_k} & \coh{1}
{\cat{M}_{lc}'(n,d)}{\cat{S}ym^{k-1} T\cat{M}_{lc}'} 
\ar[r] & \cdots }
\end{equation}

In order to prove the theorem, it is enough 
to show that 
the connecting homomorphism $\delta'_k$, 
depicted in the above commutative diagram \eqref{eq:cd2}, is injective for all $k
\geq 1$.
Again from the above commutative diagram \eqref{eq:cd2},
$\delta'_k$ is injective for every $k \geq 1$
if and only if 
 the connecting homomorphism 
\begin{equation}
\label{eq:4.d}
\delta_k: \coh{0}{\cat{M}_{lc}'(n,d)}{\cat{S}ym^k T\cat{M}_{lc}'}      \to  \coh{1}{\cat{M}_{lc}'(n,d)}
{\cat{S}ym^{k-1} T \cat{M}_{lc}'}
\end{equation}
is injective for every $k \geq 1$.

Let $\at{\zeta} \in \coh{1}
{\cat{M}_{lc}'(n,d)}{ T^*\cat{M}_{lc}'} $ be the Atiyah
class of the line bundle $\zeta$, which is nothing but
the extension class of the Atiyah exact sequence
(see \cite{A})
\begin{equation}
\label{eq:33}
0 \to \struct{\cat{M}_{lc}'} \to \cat{D}^1 (\zeta) \xrightarrow{ \sigma_1} T\cat{M}_{lc}' \to 0.
\end{equation}
The Atiyah class $\at{\zeta}$ determines the first 
Chern class $c_1(\zeta)$ of the line bundle $\zeta$.
Let $\gamma_k$ be the extension class of the short
exact sequence \eqref{eq:29}. Since the short 
exact sequence \eqref{eq:29} is the symmetric power 
of \eqref{eq:33}, the extension class $\gamma_k$
can be expressed in terms of the first Chern class
$c_1(\zeta)$. Further, let $\alpha_k$ denote the 
extension class
of the following short exact sequence 
\begin{equation}
\label{eq:34}
0 \to  \cat{S}ym^{k-1}(T\cat{M}_{lc}') \to
\frac{\cat{S}ym^k(\cat{D}^1 (\zeta))}{\cat{S}ym^{k-2}(\cat{D}^1 (\zeta))} \to \cat{S}ym^k(T \cat{M}_{lc}') \to 0, 
\end{equation}
which is the bottom short exact sequence in the commutative
diagram \eqref{eq:cd1}. Then $\gamma_k$ maps to 
$\alpha_k$.  Thus, $\alpha_k$ can also be described 
in terms of the first Chern class $c_1(\zeta)$.

Since a connecting homomorphism can be expressed as the
cup product by the extension class of the corresponding
short exact  sequence, the connecting homomorphism 
$\delta_k$ in \eqref{eq:4.d} can be described using the 
first Chern class $c_1(\zeta)$ of the line bundle 
$\zeta$. 
Indeed, 
the cup 
product with $ c_1(\zeta)$ gives rise to a homomorphism
\begin{equation}
\label{eq:35}
\tau : \coh{0}{\cat{M}_{lc}'(n,d)}{\cat{S}ym^k T \cat{M}_{lc}'} \to \coh{1}{\cat{M}_{lc}'(n,d)}{\cat{S}ym^k T \cat{M}_{lc}' \otimes T^*\cat{M}_{lc}'}.
\end{equation}
The canonical homomorphism
\begin{equation*}
\label{eq:36}
\upsilon:\cat{S}ym^k T\cat{M}_{lc}' \otimes T^*\cat{M}_{lc}' \to  \cat{S}^{k-1}T \cat{M}_{lc}'
\end{equation*}
 induces a morphism of  \C-vector spaces
\begin{equation}
\label{eq:37}
\upsilon^*:\coh{1}{\cat{M}_{lc}'(n,d)}{\cat{S}ym^k T \cat{M}_{lc}' \otimes T^*\cat{M}_{lc}'} \to \coh{1}
{\cat{M}_{lc}'(n,d)}{\cat{S}ym^{k-1} T\cat{M}_{lc}'}.
\end{equation}
Thus, we get a morphism
\begin{equation}
\label{eq:38}
\tilde{\tau} = \upsilon^* \circ \tau: \coh{0}{\cat{M}_{lc}'(n,d)}{\cat{S}ym^k T 
\cat{M}_{lc}'} \to \coh{1}{\cat{M}_{lc}'(n,d)}
{\cat{S}ym^{k-1} T\cat{M}_{lc}'},
\end{equation}
Then from the above observation we have $\tilde{\tau} = \delta_k$. 
It is sufficient to show that $\tilde{\tau}$ is injective. In view of the assumption that ${\widetilde{p_0}}_{\sharp}$ in \eqref{eq:41.15} is an injective morphism, now
using the similar technique as in the proof of \cite[Theorem 1.4]{AS}, we can show that $\tilde{\tau}$ is injective.

\end{proof}

\begin{proposition}
\label{prop:4.1} Under the hypothesis of the Theorem
\ref{thm:4.1}, for $k \geq 0$, we have 
\begin{equation}
\label{pr}
 \coh{0}{\cat{M}_{lc}'(n,d)}{\cat{D}^k (\zeta)} = \C.
 \end{equation}
\end{proposition}
\begin{proof}
Proof follows from the similar steps as in Theorem
\ref{thm:4.1}.
\end{proof}

Under the same hypothesis of the Theorem \ref{thm:4.1} for the corresponding moduli spaces, we have

\begin{theorem}
\label{thm:log1} Suppose that the hypothesis of Theorem
\ref{thm:4.1} holds for the moduli space $\cat{X}$, where $\cat{X}$ denote $\cat{M}_{lc}'(n,L)$,
$\cat{M}'_h(n)$ or $\cat{M}'_{h}(n,L_0)$.
Let $\zeta $ be a line bundle over $\cat{X}$. Then, for every $k \geq 0$, we have 
\begin{enumerate}
\item $\coh{0}{\cat{X}}{\cat{S}ym^k(\cat{D}^1 (\zeta))} = \C.$
\item $\coh{0}{\cat{X}}{\cat{D}^k (\zeta)} = \C.$
\end{enumerate}
\end{theorem}

In \cite{BR1}, global sections of a line bundle on 
the moduli space of logarithmic connections singular
exactly over one point of a compact Riemann surface 
 have been studied. In this section, we  
 study global sections of  line bundles over $\cat{M}_{lc}'(n,L)$.

Let $\cat{L}$ be a line bundle over 
$\cat{M}_{lc}'(n,L)$. Then 
\begin{equation}
\label{eq:48}
\cat{L} = q^*\Theta^l
\end{equation}
for some $l \in \Z$, where $p$ is the morphism defined 
in \eqref{eq:12.5} and $\Theta$ is the 
generalised theta line bundle over $\cat{U}(n,L)$.
Then we have a natural generalisation of \cite[p.797, Theorem 4.3]{BR1}, and the same ideas can
be used to prove the following.
\begin{theorem}
\label{thm:4}
For every  $l < 0 $, we have 
\begin{equation}
\label{eq:49}
\coh{0}{\cat{M}_{lc}'(n,L)}{ q^*\Theta^l} = 0.
\end{equation}
\end{theorem}

\section{Torelli type theorem for the moduli spaces}
\label{Betti-mod-conn}
In \cite[Theorem 5.2]{B04}, a Torelli type theorem
has been proved for the moduli space of holomorphic 
connections over compact Riemann surface, and 
in \cite{BM7}, Torelli type theorems have been 
proved for the moduli space of logarithmic connections 
singular exactly over one point with fixed residue.
In this section, we prove Torelli type theorems for the 
moduli spaces $\cat{M}_{lc}(n,d)$ and 
$\cat{M}_{lc}(n,L)$.
We assume that 
$\Phi_j = \alpha_j \id{E(x_j)}$, for every $j = 1, \ldots, m$, where $\alpha_j \in \C$.

We show that the isomorphism classes of the moduli spaces
$\cat{M}_{lc}(n,d)$ and  
$\cat{M}_{lc}(n,L)$ do not depend on 
the choice of $S$. Let $T = \{y_1, \ldots, y_m \}$ be a finite subset of 
$X$ such that $y_i \neq y_j$ for $i \neq j$. Note that $\sharp S = \sharp T$. 

In this section, we use the following notations
$$\cat{M}_{lc}(X,S) := \cat{M}_{lc}(n,d),$$
and $$\cat{M}_{lc}(X,S,L) := \cat{M}_{lc}(n,L)$$ to 
emphasize $S$ and $T$.
Let $\cat{M}_{lc}(X,T)$ and $\cat{M}_{lc}(X,T,L)$
denote the moduli spaces corresponding to $T$. 

\begin{lemma}
\label{lem:5.1} There is an isomorphism between 
$\cat{M}_{lc}(X,S)$ and $\cat{M}_{lc}(X,T)$.
\end{lemma}
\begin{proof} Depending on the sets $S$ and $T$,
we have two cases
\begin{enumerate}
\item $S \cap T = \emptyset.$
\item $S \cap T \neq \emptyset.$
\end{enumerate}
Suppose $S \cap T = \emptyset$. For every $i = 1, \ldots, m$, 
let $L_i = \struct{X}(y_i - x_i)$ be a line bundle of 
degree zero. Let $D^i$ be the de Rham logarithmic connection on the line bundle $L_i$ singular over 
$x_i$ and $y_i$, defined by sending a local section 
$s_i$ of $L_i$ to $ds_i$.
Then $Res(D^i, x_i) = -1$ and $Res(D^i, y_i) = 1$. 
Define a line bundle 
$$L_0 = \bigotimes_{i =1}^m L_i.$$
Then $L_0$ admits a logarithmic connection induced 
from $\{D^i \}_{i = 1}^m$, which can be expressed as 
follows
$$D_0 = \sum_{i = 1}^m \id{L_1} \otimes \cdots \otimes \alpha_i {D^i} \otimes \cdots \otimes \id{L_m}.$$
Moreover, $Res(D_0, x_i) = - \alpha_i$ and $Res(D_0, y_i)
= \alpha_i$ for every $i = 1, \ldots, m$.
Let $(E, D) \in \cat{M}(X,S)$. Then
$E \otimes L_0$ admits a logarithmic connection given by
$$D \otimes \id{L_0} + \id{E} \otimes D_0.$$
Note that  for every $i = 1, \ldots, m$, we have
$$Res(D \otimes \id{L_0} + \id{E} \otimes D_0, x_i )
= 0,$$
and $$ Res(D \otimes \id{L_0} + \id{E} \otimes D_0, y_i )
= \alpha_i.$$
Thus, we have a morphism 
$$\Psi_{(L_0, D_0)} : \cat{M}(X,S) \longrightarrow \cat{M}(X,T)$$ of algebraic varieties 
sending $(E,D)$ to $(E \otimes L_0, D \otimes \id{L_0} + \id{E} \otimes D_0)$, which is an 
isomorphism.

Next suppose that $S \cap T \neq \emptyset$.                
Without loss of generality, we assume that 
$x_1 = y_1, x_2 = y_2, \ldots, x_r = y_r $ for 
$r \leq m$. In this case, we consider the line bundle 
$$L_0 = \bigotimes_{j=r+1}^{m} \struct{X}(y_j - x_j).$$
Now, using the above steps,
we can get a logarithmic connection $D_0$ in $L_0$.
A morphism similar to $\Psi_{(L_0,D_0)}$ can be defined,
which turns out to be an isomorphism. 

\end{proof}

A similar result is true for the moduli space 
$\cat{M}_{lc}(X, S, L)$. 
\begin{lemma}
\label{lem:5.2}
There is an isomorphism between 
$\cat{M}_{lc}(X,S,L)$ and $\cat{M}_{lc}(X,T, L')$.
\end{lemma}

  Thus, for the simplicity 
of the notations, we write 
$\cat{M}_{lc}(X)$ in place of $\cat{M}_{lc}(X,S)$
and $\cat{M}_{lc}(X,L)$ in place of $\cat{M}_{lc}(X,S,L)$.

Now, we shall compute the cohomology group of 
$\cat{M}_{lc}(X)$ and $\cat{M}_{lc}(X,L)$.

Let $\cat{M}'_{lc}(X) := \cat{M}'_{lc}(n,d)$,
and $\cat{M}'_{lc}(X,L) := \cat{M}'_{lc}(n,L)$. Then, 
let $$p_0 :  \cat{M}'_{lc}(X) \longrightarrow
\cat{U}^s(n,d)$$ be the morphism defined in \eqref{eq:21.5}. Since a fibre of $p_0$
is an affine space modelled over a vector space,
which is contractible, we get an isomorphism 
\begin{equation}
\label{eq:55}
p_0^*: \coh{i}{\cat{U}^s(n,d)}{ \Q}\longrightarrow \coh{i}{\cat{M}'_{lc}(X)}{ \Q}
\end{equation}
of rational cohomology groups for every $i \geq 0$.
For the cohomology of $\cat{U}^s(n,d)$ see \cite{AB} and 
\cite{K}.

Let $Z := \cat{M}_{lc}(n,L) \setminus \cat{M}_{lc}'(n,L)$. Then, form \cite[Lemma 3.1]{BM7}, we have
\begin{lemma}
\label{lem:2.5}
The codimension of the Zariski closed set $Z$ in $\cat{M}_{lc}(n,L)$ is at least $(n-1)(g-2)+1$.
In particular, if $n \geq 2$, $g \geq 3$, then 
$\mathrm{codim}(Z, \cat{M}_{lc}(n,L)) \geq 2$.
\end{lemma}

Similarly, let  $p: \cat{M}'_{lc}(X, L) \to \cat{U}^s(n, L)$ be the morphism defined in \eqref{eq:12.5}. Then $p$
 is a fibre bundle with fibres as affine spaces modelled over vector spaces, and
 since affine spaces with the usual topology are contractible,
  the induced homomorphism
 \begin{equation}
 \label{eq:56}
 p^* : \coh{i}{ \cat{U}^s(n,L)}{ \Z} \longrightarrow
  \coh{i}{\cat{M}'_{lc}(X,L)}{ \Z}
 \end{equation}
 of  cohomology  groups, is an isomorphism for all $i \geq 0$.

Let $\cat{Y}$ be a complex algebraic variety. For every
$i \geq 0$, there is a {\bf mixed Hodge structure} on the cohomology group $\coh{i}{\cat{Y}}{ \Z}$. This result is 
due to Deligne, for more details see \cite{D1}, \cite{D2}. 

The isomorphism $p^*$ in \eqref{eq:56} is an isomorphism of mixed Hodge structures. Moreover,
the cohomology group $\coh{i}{\cat{M}'_{lc}(X,L)}{ \Z}$
is equipped with {\bf pure Hodge structure} of weight $i$
for every $i \geq 0$, because $\cat{U}^s(n,L)$ is a smooth projective 
variety over $\C$, and from \cite{D1} the cohomology group $\coh{i}{ \cat{U}^s(n,L)}{ \Z}$ is endowed with a pure Hodge structure of weight 
$i$, for every $i \geq 0$.

Let $\cat{A}$ be a smooth complex analytic space.
For every integer $k \geq 0$, the $(k+1)$-th {\bf intermediate Jacobian variety} $J^{k+1}(\cat{A})$
of $\cat{Y}$ is defined as follows.
\begin{equation}
\label{eq:57}
J^{k+1}(\cat{A}) := \coh{2k+1}{ \cat{A}}{ \R} / \coh{2k+1}{ \cat{A}}{ \Z}
\end{equation}
The space $J^{k+1}(\cat{A})$ carries a canonical structure of complex manifold.
We consider that the moduli space $\cat{M}_{lc}(X,L)$
is equipped with the complex analytic topology.
\begin{proposition} 
\label{prop:5}
The second intermediate 
$J^{2}(\cat{M}_{lc}(X,L))$ is isomorphic to the Jacobian
$J(X) := Pic^{0}(X)$ of $X$.
\end{proposition}
\begin{proof} First we show that the mixed Hodge structure on $\cat{M}_{lc}(X,L)$ is in fact a pure 
Hodge structure.
Let $Z := \cat{M}_{lc}(X,L) \setminus \cat{M}'_{lc}(X,L)$ as in Lemma \ref{lem:2.5}. Then, we have a long 
exact sequence of relative cohomology groups,
$$ \RCOH[3]{Z}{\cat{M}_{lc}(X,L)}{\Z} \to \coh{3}{ \cat{M}_{lc}(X,L)}{ \Z} \xrightarrow{\iota^*} \coh{3}{ \cat{M}'_{lc}(X,L)}{ \Z} \xrightarrow{\partial}
\RCOH[4]{Z}{\cat{M}_{lc}(X,L)}{\Z}$$
where $\iota^*$ is induced by the inclusion map
$\iota : \cat{M}'_{lc}(X,L) \hookrightarrow \cat{M}_{lc}(X,L)$ and $\partial$ is the boundary operator.
We show that $\iota^*$ is an isomorphism.
From Alexander duality \cite[Theorem 4.7, p.381]{I}, we have an isomorphism 
$$\RCOH[i]{Z}{\cat{M}_{lc}(X,L)}{\Z} \longrightarrow \RCOH[BM]{2N-i}{Z}{\Z},$$
where $N = 2(n^2-1)(g-1)$ is the complex dimension of the moduli space $\cat{M}_{lc}(X,L)$ and 
$\mbox{H}^{BM}_{*}$ is the Borel-Moore homology.
In view of Lemma \ref{lem:2.5}, we have 
$$\mathrm{codim}(Z, \cat{M}_{lc}(X,L)) \geq 2,$$
therefore the real dimension of $Z$ is at most $2N-4$.
Thus, 
$$\RCOH[BM]{2N-i}{Z}{\Z} = 0, ~~\mbox{for}~i = 0,1,2,3,$$
and hence 
$$\RCOH[3]{Z}{\cat{M}_{lc}(X,L)}{\Z} = 0.$$
This implies that $\iota^*$ is an injective morphism.
Let $\Gamma$  be a smooth compactification of $\cat{M}_{lc}(X,L)$, and 
$$Z' = \Gamma \setminus \cat{M}'_{lc}(X,L).$$
Then, from \cite[Corollaire 3.2.17]{D1}, we have a
surjective morphism 
$$\coh{3}{ \Gamma}{ \Q} \longrightarrow \coh{3}{ \cat{M}_{lc}(X,L)}{ \Q} $$
of mixed Hodge structures, and since 
$$\mathrm{codim}(Z', \Gamma) \geq 3$$
from \cite[p.269, Lemma 11.13]{Vo1}, we 
have an isomorphism 
$$\coh{3}{ \Gamma}{ \Z} \longrightarrow  \coh{3}{ \cat{M}'_{lc}(X,L)}{ \Z}$$ 
of Hodge structures.
Then we 
 have a 
commutative diagram 
\begin{equation*}
\xymatrix{
\coh{3}{ \Gamma}{ \Q}  \ar[d] \ar[rd] \\
\coh{3}{ \cat{M}_{lc}(X,L)}{ \Q}  \ar[r]^{\iota^*_{\Q}} & \coh{3}{ \cat{M}'_{lc}(X,L)}{ \Q} \\
}
\end{equation*}
and from the above facts, the vertical and diagonal arrows are the surjective morphisms 
of mixed Hodge structures  induced from their respective 
inclusion maps.  Now, because of the commutativity of the diagram, $\iota^*_{\Q}$ is a surjective morphism of mixed Hodge structures.
Since $\coh{3}{ \cat{M}'_{lc}(X,L)}{ \Z}$ (being isomorphic to $\coh{3}{ \cat{U}(n,L)}{ \Z}$) is torsion 
free $\Z$-module of finite rank \cite[Theorem 3]{NR1} and  $\iota^*$ is an injective morphism, $\coh{3}{ \cat{M}_{lc}(X,L)}{ \Z}$
is torsion free. In order to show that $\iota^*$ is 
surjective, we need to show that $\RCOH[4]{Z}{\cat{M}_{lc}(X,L)}{\Z}$ is torsion-free. In fact, if $Z$ has
codimension $\geq 3$, this group is zero; if $Z$ has 
codimension $2$, it is isomorphic to $\RCOH[BM]{2N-4}{Z}{\Z}$, which is the top homology group and necessarily
torsion-free.
Thus, $\iota^*$ is a surjective morphism, and hence 
the mixed Hodge structure on
$\coh{3}{ \cat{M}_{lc}(X,L)}{ \Z}$ is a pure Hodge 
structure of weight $3$.
Therefore, the second intermediate Jacobian 
$J^{2}(\cat{M}_{lc}(X,L))$ is isomorphic to 
$J^{2}(\cat{M}'_{lc}(X,L))$, and the latter is isomorphic 
to $J^{2}(\cat{U}(n,L))$.
Thus, from \cite[Theorem 3]{NR1}, $J^{2}(\cat{M}_{lc}(X,L))$ is isomorphic to $J(X)$. This completes the proof.
\end{proof}
 
Let $\Theta$ be the theta divisor on the Jacobian 
$J(X)$. The pair $(J(X), \Theta)$ is called a principally 
polarised Jacobian. Then, the classical Torelli  theorem says that the pair $(J(X), \Theta)$ determines the
 compact Riemann surface $X$ up to isomorphism.
 
In view of  Proposition \ref{prop:5}, the moduli space  $\cat{M}_{lc}(X,L)$ determines the Jacobian $J(X)$ of 
the compact Riemann surface $X$. But this does not 
qualify for the determination of $X$, because two 
non-isomorphic compact Riemann surfaces can have 
isomorphic Jacobian.

Nevertheless, from \cite[p.125, Corollary 1.2]{NN1},
there are, up to isomorphism, only finitely many compact 
Riemann surfaces having a given abelian variety as 
the Jacobian. Thus, 
 there are, up to isomorphism, only finitely many 
 compact Riemann surface $Y$ such that 
 $\cat{M}_{lc}(Y,L)$ is isomorphic to $\cat{M}_{lc}(X,L)$.

 \begin{remark}
 \label{rmk:1}
Let $\widetilde{\Theta}$ be the canonical  polarisation on  the second intermediate Jacobian  $J^{2}(\cat{U}(n,L))$. Then, from \cite[Theorem 3]{NR1}, we  have 
$$(J^{2}(\cat{U}(n,L)), \widetilde{\Theta}) \cong 
(J(X), \Theta).$$
 In \cite[Section 4]{BM7},   Biswas and Mu$\tilde{\mbox{n}}$oz constructed the principal polarisation $\widehat{\Theta}$ on the second intermediate Jacobian of the moduli space $\cat{M}^{x_0}_{lc}(X)$ of logarithmic connections singular exactly over one point $x_0$ of the compact Riemann surface $X$
with fixed determinant such that the principally polarised abelian variety 
$(J^2(\cat{M}^{x_0}_{lc}(X)), \widehat{\Theta})$ is isomorphic to the principally polarised abelian 
variety $(J^2(\cat{U}(n,L)), \widetilde{\Theta})$.
Imitating the similar technique as in \cite[Section 4]{BM7}, a principal polarisation can be constructed on $\cat{M}_{lc}(X,L)$. 
 \end{remark}
 
From Lemma \ref{lem:5.2},
 the moduli space $\cat{M}_{lc}(X,L)$ does not depend on the choice of $S$.
Thus, we have 
\begin{theorem}
\label{thm:5.5}
Let $(X,S)$ and $(Y, T)$ be two $m$-pointed compact Riemann surfaces of genus 
$g \geq 3$. Let $\cat{M}_{lc}(X,L)$ and $\cat{M}_{lc}(Y,L')$ be the corresponding moduli spaces of 
logarithmic connections. Then, $\cat{M}_{lc}(X,L)$ is isomorphic to $\cat{M}_{lc}(Y,L')$  if and only if $X$ is 
isomorphic to $Y$.
\end{theorem}

Next, we show the Torelli type theorem for the moduli space
$\cat{M}_{lc}(X)$.
Let 
\begin{equation}
\label{eq:5.6}
G : \cat{M}_{lc}(X) \longrightarrow Pic^d(X)
\end{equation}
be the  map sending 
$(E,D) \mapsto \bigwedge^nE$. Note that the morphism
$G$ is surjective.    Since $d$, $n$ and 
$\Phi_j = \alpha_j \id{E(x_j)}$ for $j = 1, \dots, m$ satisfy \eqref{eq:8},  
from \cite[Theorem 1.3 (2)]{B} $E$ admits a logarithmic 
connection with residues $\alpha_j \id{E(x_j)}$ at $x_j \in S$.

Now, we have a natural generalisation of \cite[p.431, Lemma 5.1]{B04} and \cite[p.313, Proposition 5.1]{BM7}.
\begin{proposition}
\label{prop:5.7}
Let $A$ be a complex abelian variety, and 
\begin{equation}
\label{eq:5.8}
f : \cat{M}_{lc}(X) \longrightarrow A
\end{equation}
a regular morphism. Then there exists a unique
regular morphism
$$f_0 : Pic^d(X) \longrightarrow A$$
such that 
\begin{equation}
\label{eq:5.9}
f_0 \circ G = f,
\end{equation}
where $G$ is defined in \eqref{eq:5.6}.
\end{proposition}

\begin{proof}
Consider $\cat{M}'_{lc}(X) := \cat{M}'_{lc}(n,d) \subset \cat{M}_{lc}(X)$ as in  \eqref{eq:m1}.
Let $$p_0 : \cat{M}'_{lc}(X) \longrightarrow \cat{U}(n,d) $$
be the morphism defined in \eqref{eq:21.5}. For 
$E \in \cat{U}(n,d)$,
it has been observed that $p_0^{-1}(E)$ is an affine 
space modelled over the vector space $\coh{0}{X}{\Omega^1_X \otimes \ENd{E}}$, and hence $p_0^{-1}(E)$
is a rational variety. Restricting $f$ to $p_0^{-1}(E)$, we get a map
$$f \vert_{p_0^{-1}(E)} :   p_0^{-1}(E) \longrightarrow A,$$
which is a constant map, because any regular morphism from 
a rational variety to an abelian variety is constant.

Now, consider the determinant map
$$F : \cat{U}(n,d) \longrightarrow Pic^{d}(X)$$
defined by sending $E$ to $\bigwedge^n E$. Then $F$ 
is a surjective map. For any $L \in Pic^d(X)$,
$F^{-1}(L)$ is nothing but the moduli space $\cat{U}(n,L)$. 
Thus, we get a regular morphism
$$\psi_0 \vert_{F^{-1}(L)} : \cat{U}(n,L) = F^{-1}(L) \longrightarrow A.$$
on each of the fibres of $F$.
From \cite[Theorem 1.2]{KS}, $\cat{U}(n,L)$ is a rational variety, and hence the regular morphism 
$\psi_0 \vert_{F^{-1}(L)}$ is constant.
This completes the proof.
\end{proof}

Let $\cat{M}_{lc}^X$ be the moduli space defined in 
\eqref{eq:k}.
Then, we have a morphism
\begin{equation}
\label{eq:5.10}
\delta: \cat{M}_{lc}^X \longrightarrow Pic^d(X)
\end{equation}
defined by $(L,D) \mapsto L$. Then 
$\delta^{-1}(L)$ is an affine space modelled over 
$\coh{0}{X}{\Omega^1_X}$.
Then 
$$G = \delta \circ \det,$$
where $G$ is defined in \eqref{eq:5.6}, and 
$\det : \cat{M}_{lc}(X) \to \cat{M}_{lc}^X$  defined in \eqref{mor}.
Thus, we have a morphism
\begin{equation}
\label{eq:5.11}
\eta : G^{-1}(L) \longrightarrow \cat{M}_{lc}(X,L)
\end{equation}
which is a fibration and each fibre is an affine space 
modelled over $\coh{0}{X}{\Omega^1_X}$.
Since the fibre of $\eta$ is contractible, we have 
an isomorphism
$$\eta^* : \coh{i}{ \cat{M}_{lc}(X,L)}{ \Z} \longrightarrow
  \coh{i}{G^{-1}(L)}{ \Z} $$
  of cohomology groups for all $i \geq 0$.
  Therefore, we have
  $$J^2(\cat{M}_{lc}(X,L)) \cong J^2(G^{-1}(L)).$$
  As mentioned in  Remark \ref{rmk:1}, similar steps
  give a principal polarisation $\widehat{\widehat{\Theta}}$ on $J^2(G^{-1}(L))$
  such that 
   $$(J^2(\cat{M}_{lc}(X,L)), \widehat{\Theta}) \cong (J^2(G^{-1}(L)), \widehat{\widehat{\Theta}}).$$
   
   Thus, in view of Lemma \ref{lem:5.1} and  using the Theorem \ref{thm:5.5}, we get
   \begin{theorem}
   \label{thm:5.9}
Let $(X,S)$ and $(Y, T)$ be two $m$-pointed compact Riemann surfaces of genus 
$g \geq 3$. Let $\cat{M}_{lc}(X)$ and $\cat{M}_{lc}(Y)$ be the corresponding moduli spaces of 
logarithmic connections. Then, $\cat{M}_{lc}(X)$ is isomorphic to $\cat{M}_{lc}(Y)$  if and only if $X$ is 
isomorphic to $Y$.   
   \end{theorem}
  
%

\section{Rational connectedness of the moduli spaces}
\label{Rat}
 In \cite{AS1},
we have shown that the moduli space of rank one logarithmic connections with fixed residues is not 
rational. In this section, we show that the moduli space 
$\cat{M}_{lc}(n,d)$ is not rational. For the theory of 
rational varieties, we refer to \cite{Ko}. 

Recall that  
 a smooth complex variety $V$  
is said to be rationally connected if any two general points on $V$
can be connected by a rational curve in $V$. The following lemma is an easy consequence of  the definition.

\begin{lemma}
\label{lem:6.1}
Let $f : \cat{Y} \to \cat{X}$ be a dominant rational map of complex 
algebraic varieties with $\cat{Y}$ rationally connected. Then,  $\cat{X}$ is rationally connected.
\end{lemma}

\begin{theorem}[\cite{KS}, Theorem 1.1] 
The moduli space 
$\cat{U}(n,d)$ is birational to $J(X) \times \A^{(n^2-1)(g-1)}$, where $J(X)$ is the Jacobian of $X$.
\end{theorem}

Note that $J(X)$ is not rationally connected,
because it does not contain any rational curve. 
Therefore, $\cat{U}(n,d)$ is not rationally connected.

\begin{proposition}
\label{thm:6.3}
The moduli space $\cat{M}_{lc}(n,d)$ is not rational.
\end{proposition}
\begin{proof}
It is enough to show that the moduli space $\cat{M}_{lc}(n,d)$ is not rationally connected. Let 
$$p_0: \cat{M}_{lc}'(n,d) \longrightarrow \cat{U}(n,d)$$
be the morphism of varieties defined in \eqref{eq:21.5}.
Suppose that $\cat{M}'_{lc}(n,d)$ is rationally connected. Then, from Lemma \ref{lem:6.1}, 
$\cat{U}(n,d)$ is rationally connected, which is not true. Thus, $\cat{M}'_{lc}(n,d)$ is not rationally connected and hence not rational.
Since $\cat{M}'_{lc}(n,d)$ is an open dense subset of 
$\cat{M}_{lc}(n,d)$, $\cat{M}_{lc}(n,d)$ is not rational.
\end{proof}

A similar argument gives the following.
\begin{proposition}
\label{thm:6.3.1}
The moduli space $\cat{M}_h(n)$ is not rational.
\end{proposition}

\begin{lemma}[\cite{GHS}, Corollary 1.3]
\label{lem:6.4}
Let $f: \cat{X} \to \cat{Y}$ be any dominant morphism 
of complex varieties. If $\cat{Y}$ and the general 
fibre of $f$ are rationally connected, then $\cat{X}$
is rationally connected.
\end{lemma}

\begin{proposition}
\label{prop:6.5}
The moduli space $\cat{M}'_{lc}(n,L)$ is rationally connected.
\end{proposition}
\begin{proof}
Consider the dominant  morphism 
$$p: \cat{M}'_{lc}(X,L) \longrightarrow \cat{U}(n,L)$$
defined in \eqref{eq:12.5}. As observed earlier every 
fibre of $p$ is an affine space and hence rationally 
connected. Since $\cat{U}(n,L)$ is rationally connected,  
from Lemma \ref{lem:6.4},  $\cat{M}'_{lc}(n,L)$ is 
rationally connected.
\end{proof}

\begin{corollary}
\label{cor:6.7}
$\cat{M}_{lc}(n,L)$ is rationally connected.
\end{corollary}
\begin{proof}
It follows from the fact that 
rationally connectedness is a birational invariant, and $\cat{M}'_{lc}(n,L)$ is a dense open subset of 
 $\cat{M}_{lc}(n,L)$.
\end{proof}

Therefore, we have  a natural question.
\begin{question}
\label{q:1}
Is the moduli space $\cat{M}_{lc}(n,L)$ rational ?
\end{question}

\section*{Acknowledgements} 
The author would like to thank the anonymous referees for 
helpful comments, numerous suggestions and corrections.


\end{document}